\documentclass[11pt,leqno,oneside]{amsart}
\usepackage{layout}
\usepackage{mathrsfs,dsfont}
\usepackage{amsmath,amstext,amsthm,amssymb,bbm,color}
\usepackage{comment}
\usepackage{charter}
\usepackage{typearea}
\usepackage{pdfsync}
\usepackage[width=6.5in,height=8.7in]{geometry}

\def\bt{\begin{thm}}
\def\et{\end{thm}}
\def\bl{\begin{lem}}
\def\el{\end{lem}}
\def\bd{\begin{defn}}
\def\ed{\end{defn}}
\def\bc{\begin{cor}}
\def\ec{\end{cor}}
\def\bp{\begin{proof}}
\def\ep{\end{proof}}
\def\br{\begin{rem}}
\def\er{\end{rem}}

\newtheorem{thm}{Theorem}[section]
\newtheorem{prop}[thm]{Proposition}
\newtheorem{lem}[thm]{Lemma}
\newtheorem{defn}[thm]{Definition}
\newtheorem{example}[thm]{Example}
\newtheorem{rem}[thm]{Remark}
\newtheorem{cor}[thm]{Corollary}

\numberwithin{equation}{section}

\newcommand{\C}{\Bbb{C}^m}
\newcommand{\R}{\Bbb{R}}
\newcommand{\E}{\Bbb{E}}
\newcommand{\pn}{\mathcal{P}_n}

\newcommand{\bthm}{\begin{thm}}
\newcommand{\ethm}{\end{thm}}
\newcommand{\bstp}{\begin{stp}}
\newcommand{\estp}{\end{stp}}
\newcommand{\blemma}{\begin{lemma}}
\newcommand{\elemma}{\end{lemma}}
\newcommand{\bprop}{\begin{prop}}
\newcommand{\eprop}{\end{prop}}
\newcommand{\bpf}{\begin{pf}}
\newcommand{\epf}{\end{pf}}
\newcommand{\bdefn}{\begin{defn}}
\newcommand{\edefn}{\end{defn}}
\newcommand{\brk}{\begin{rmrk}}
\newcommand{\erk}{\end{rmrk}}
\newcommand{\bcrl}{\begin{crl}}
\newcommand{\ecrl}{\end{crl}}

\title[Mass equidistribution]{ Mass equidistribution for random polynomials}

\author{Turgay Bayraktar} 
\thanks{T.\ Bayraktar is partially supported by T\"{U}B\.{I}TAK grants B\.{I}DEB-2232/118C006, ARDEB-3501/118F049 and Science Academy, Turkey BAGEP grant.}
\address{Faculty of Engineering and Natural Sciences, Sabanc{\i} University, \.{I}stanbul, Turkey}
\email{tbayraktar@sabanciuniv.edu}

\date{\today}

\keywords{Random polynomial, equidistribution of zeros, equilibrium measure, global extremal function, Bergman kernel asymptotics}
\subjclass[2000]{32A60,32A25,60D05}

\begin{document}

\begin{abstract}
 The purpose of this note is to study asymptotic zero distribution of multivariate random polynomials as their degrees grow. For a smooth weight function with super logarithmic growth at infinity, we consider random linear combinations of associated orthogonal polynomials with subgaussian coefficients. This class of probability distributions contains a wide range of random variables including standard Gaussian and all bounded random variables. We prove that for almost every sequence of random polynomials their normalized zero currents  become equidistributed with respect to a deterministic extremal current. The main ingredients of the proof are Bergman kernel asymptotics, mass equidistribution of random polynomials and concentration inequalities for subgaussian quadratic forms. 
  \end{abstract}
\maketitle


\section{Introduction}

Let $\varphi:\C\to \Bbb{R}$ be a $\mathscr{C}^{1,1}$ \textit{weight function} (i.e.\ $\varphi$ is differentiable
and all of its first partial derivatives are locally Lipschitz continuous) satisfying  
 \begin{equation}\label{growth}
\varphi(z)\geq (1+\epsilon)\log\|z\|\ \text{for}\ \|z\|\gg 1
\end{equation} for some fixed $\epsilon>0.$ 
We define an inner product on the space $\pn$ of multi-variable polynomials of degree at most $n$ by setting 
\begin{equation}\label{in}
\langle p,q\rangle_n:=\int_{\C}p(z)\overline{q(z)}e^{-2n\varphi(z)}dV_m(z)
\end{equation} where $dV_m$ denotes the Lebesgue measure on $\C$. We also let $\{P_j^n\}_{j=1}^{d_n}$ be a fixed orthonormal basis (ONB) for $\mathcal{P}_n$ with respect to the inner product (\ref{in}).
A \textit{random polynomial} is of the form
$$f_n(z)=\sum_{j=1}^{d_n}c^n_jP_j^n(z)$$ where $c_j^n$ are independent identically distributed (iid) real or complex subgaussian random variables (see \S\ref{gauss1}) and $d_n:=\dim(\mathcal{P}_n)={n+m\choose n}$.  This allows us to endow $\pn$ with a $d_n$-fold product probability measure $Prob_n$ induced by the probability law of $c_j^n$. We also consider the product probability space $\prod_{n=1}^{\infty}(\pn, Prob_n)$ whose elements are sequences of random polynomials of increasing degree. We are interested in limiting distribution of zeros of random polynomials.

 In the present setting, the choice of weight function $\varphi$ determines a weighted global extremal function $\varphi_e$ (see \ref{exfunc}) which induces a weighted equilibrium measure $\mu_e$ (see \ref{eqm}) whose support is a compact set denoted by $S_{\varphi}$. The following result indicates that for a typical (in the sense of probability) sequence $\{f_n\}_{n=1}^{\infty}$ of random polynomials the masses (respectively, normalized zero currents) are asymptotic to the equilibrium measure (respectively, the extremal current):

\begin{thm}\label{main}
Let $\varphi:\C\to \Bbb{R}$ be a $\mathscr{C}^{1,1}$-weight function satisfying (\ref{growth}). Assume that random coefficients $c_j^n$ are iid real or complex subgaussian random variables of mean zero and unit variance. Then almost surely in $\prod_{n=1}^{\infty}(\pn, Prob_n)$ the masses 
\begin{equation}
\frac{1}{d_n}|f_n(z)|^2e^{-2n\varphi(z)}dV_m \to d\mu_{\varphi_e}
\end{equation} in the weak-star topology of measures on $S_{\varphi}$.
 Moreover, almost surely in $\prod_{n=1}^{\infty}(\pn, Prob_n)$ the normalized currents of integrations
 $$\frac1n[Z_{f_n}] \to dd^c\varphi_e$$ in the sense of currents.
\end{thm}

Distribution of zeros of random polynomials is a classical subject which goes back to Kac \cite{Kac} and Hammersley \cite{Ham} among others. A classical result due to Kac and Hammersley asserts that normalized zeros of Kac random polynomials (i.e.\ $\sum_{j=0}^nc_jz^j$ with iid Gaussian coefficients) of large degree tend to accumulate on the unit circle $S^1=\{|z|=1\}.$ This ensemble of random polynomials has been extensively studied (see eg. \cite{LO,ET,SV,HN,IZ,Pritsker} and references therein). In \cite{SZ}, Shiffman and Zelditch obtained a far reaching generalization of the aforementioned result in the line bundle setting. Following \cite{SZ}, asymptotic distribution of zero divisors of multi-variable random polynomials with random coefficients has been studied by various authors. In \cite{Bloom1, BloomS} Bloom and Shiffman (see also \cite{B4}) considered random polynomials with Gaussian coefficients. Random polynomials with non-Gaussian coefficients were also considered by various authors (see eg. \cite{DS3,BloomL,B6,B7, BloomD} among others). In \cite{B10} for radially symmetric weight functions, we provided a necessary and sufficient condition on random coefficients for equilibrium distribution of zero divisors of random polynomials (see also \cite{BCM} for the line bundle setting). We refer the reader to the recent survey \cite[\S2]{BCHM} for the state of the art. 

Mass asymptotics have been considered by several authors in various geometric settings. Given a compact K\"ahler manifold $(M,\omega)$ of dimension $m$ and a positive Hermitian holomorphic line bundle $(L,h)$ whose curvature form $c_1(L,h)=\omega$, one can define a scalar $L^2$-product and a norm on the vector space of global holomorphic sections $H^0(X,L^{\otimes n})$ by
$$\|s\|_n^2:=\int_{M}|s(x)|^2_{h^{\otimes n}}dV$$ where $dV$ is the probability volume form induced by $\omega$.
 In \cite{SZ} Shiffman and Zelditch proved that for a sequence $s_n\in H^0(X,L^{\otimes n})$ of global holomorphic sections of unit norm if their \textit{masses}
 $$|s_n(x)|^2_{h^{\otimes n}}dV\to dV$$ in the weak-star topology of measures on $M$, the normalized zero currents $\frac{1}{n}[Z_{s_n}]$ (along zero divisors of $s_n$) converge weakly to the curvature form $\omega$ (see \S\ref{lineb} for details). This was first observed by Nonnenmacher and Voros \cite{NV} in the case of the theta bundle over an elliptic curve $\Bbb{C}/\Bbb{Z}^2$. In a different direction, Rudnick \cite{Rudnick} proved a similar result in the setting of $SL_2(\Bbb{Z})$ modular cusp forms of weight $2n$. The latter corresponds to the case of positive line bundle on a non-compact Riemann surface. All of the aforementioned results are based on potential theory. 

In the $\C$ setting the result of \cite{SZ} corresponds to the case where the weight function is a K\"ahler potential (e.g. $\varphi(z)=\frac12\log[1+\|z\|^2]$). The later implies that the weighted equilibrium measure is the restriction of a volume form defined on the complex projective space $\Bbb{P}^m$.  More recently, Zelditch \cite{Zel} obtained a generalization to partially positive metrics on positive line bundles. However, the argument in \cite[Theorem 1.2]{Zel} has a gap. Namely, it  proves the $L^1_{loc}$ convergence of the potentials only in the support of the Monge-Amp\`ere measure. In order to complete the proof one needs to use a generalized domination principle (Theorem \ref{domp}). We adapt the argument in \cite{Zel} to the current setting and fill this gap. Moreover, we generalize the mass equidistribution of random polynomials with subgaussian coefficients by using Hanson-Wright inequality \cite{RuVe} for subgaussian quadratic forms.

The outline of the paper is as follows: In \S\ref{S2} we prove that mass asymptotics implies equilibrium distribution of zeros (Theorem \ref{basic}). In \S\ref{S3}, we review near and off diagonal Bergman kernel asymptotics in the special case $Y=\C$ and a $\mathscr{C}^{1,1}$-weight function $\varphi:\C\to\Bbb{R}$ that has super logarithmic growth at infinity and apply these results to study associated Toeplitz operators and distribution of their eigenvalues.  In \S\ref{S4}, we prove Theorem \ref{main}. In $\S\ref{S5}$ we discuss analogous results for random orthonormal bases. Finally, in \S\ref{lineb} we give a generalization of Theorem \ref{main} to the line bundle setting.

\section{Mass Asymptotics of Weighted Polynomials}\label{S2}
Let $Y\subset \C$ be a closed set and $\varphi:Y\to \Bbb{R}$ be a continuous \textit{weight function}. If $Y$ is unbounded we assume that  there exists $\epsilon>0$ such that
\begin{equation}\label{gr}
\varphi(z)\geq (1+\epsilon)\log\|z\|\ \text{for}\ \|z\|\gg1.
\end{equation}
Following \cite[Appendix B]{SaffTotik} we denote the \textit{weighted global extremal function}
\begin{equation}\label{exfunc}
V_{Y,\varphi}(z):=\sup\{u(z): u\in \mathcal{L}(\C), u\leq \varphi\ \text{on}\ Y\}
\end{equation}
where $\mathcal{L}(\C)$ denotes the \textit{Lelong class} of pluri-subharmonic (psh) functions $u$ that satisfies $$u(z)-\log^+\|z\|=O(1)$$ where $\log^+=\max(\log,0).$ We remark that when $Y$ is compact and $\varphi\equiv0$ (i.e.\ in the unweighted case) the extremal function defined in (\ref{exfunc}) is the pluri-complex Green function of $Y$ (cf.\ \cite{Klimek}) and denoted by $V_Y$. We also denote by
$$\mathcal{L}^+(\Bbb{C}^m):=\{u\in \mathcal{L}(\Bbb{C}^m): u(z)\geq \log^+\|z\|+C_u\ \text{for some}\ C_u\in \Bbb{R}\}.$$
In what follows, we let 
 $$g^*(z):=\limsup_{w\to z}g(w)$$
denote the upper semi-continuous regularization of $g$. Seminal results of Siciak and Zaharyuta (see \cite[Appendix B]{SaffTotik} and references therein) assert that $V^*_{Y,\varphi}\in\mathcal{L}^+(\Bbb{C}^m)$ and that  $V_{Y,\varphi}$ verifies
 \begin{equation}\label{envelope}
 V_{Y,\varphi}(z)=\sup\{\frac{1}{\deg p}\log|p(z)|:p\ \text{is a polynomial and}\ \sup_{z\in Y}|p(z)|e^{-deg(p)\varphi(z)}\leq 1\}.
 \end{equation}
 For $r>0$ let us denote $Y_r:=\{z\in Y:\|z\|\leq r\}$. It is well-known that $V_{Y,\varphi}=V_{Y_r,\varphi}$ for sufficiently large $r$ (\cite[Appendix B, Lemma 2.2]{SaffTotik}). 
 
 A closed set $Y\subset \C$ is said to be \textit{locally regular at} $w\in Y$  if for every $\rho>0$ the extremal function $V_{Y\cap \overline{B(w,\rho)}}(z)$ is continuous at $w$. The set $Y$ is called locally regular if $Y$ is locally regular at every $w\in Y$. A classical result of Siciak \cite{Siciak} asserts that if $Y$ is locally regular and $\varphi$ is continuous weight function then the weighted extremal function $V_{Y,\varphi}$ is also continuous and hence $V_{Y,\varphi}=V_{Y,\varphi}^*$ on $\C$. In the rest of this section we assume that $Y$ is a locally regular closed set.

The psh function $V_{Y,\varphi}$ is locally bounded on $\C$ and hence by Bedford-Taylor theory \cite{BT1,BT2} the \textit{weighted equilibrium measure} $$\mu_{Y,\varphi}:=\frac{1}{m!}(dd^cV_{Y,\varphi})^m$$ is well-defined and does not put any mass on pluripolar sets. Here; $d=\partial+\overline{\partial}$ and $d^c:=\frac{i}{2\pi}(\overline{\partial}-\partial)$ so that $dd^c=\frac{i}{\pi}\partial\overline{\partial}$ and 
\begin{equation}\int_{\C}\big(dd^c\frac12\log[1+\|z\|^2]\big)^m=1.\end{equation}
Moreover, denoting the support $S_{Y,\varphi}:=supp(\mu_{Y,\varphi})$ by \cite[Appendix B]{SaffTotik} we have $$S_{Y,\varphi}\subset \{z\in Y: V_{Y,\varphi}(z)=\varphi(z)\}.$$ Thus, the support $S_{Y,\varphi}$ is a compact set.  We denote its interior (as a subset of $\C$) by $Int(S_{Y,\varphi})$. An important example is $\varphi(z)=\frac{\|z\|^2}{2}$ which in turn gives $\mu_{Y,\varphi}=\mathbbm{1}_{B}dV_m$ where $\mathbbm{1}_{B}$ denotes the characteristic function of the unit ball in $\C$.

A locally finite measure $\nu$ is called a Bernstein-Markov (BM) measure for the weighted set $(Y,\varphi)$ if
for sufficiently large $r>0$ the triple $(Y_r,Q, \nu)$ satisfies the weighted Bernstein-Markov inequality. That is, there is $M_n\geq 1$ such that $\displaystyle\limsup_{n\to\infty} M_n^{1/n}=1$ and  
\begin{equation}\label{bm}
 \|pe^{-n\varphi}\|_{Y_r}:=\max_{z\in Y_r}|p(z)|e^{-n\varphi(z)}\leq M_n \|pe^{-n\varphi}\|_{L^2(\nu)}\ \forall p\in\pn.
\end{equation}
 If $Y$ is an unbounded, we also require
\begin{equation}\label{gr2}
 \int_{Y\setminus Y_r}\frac{1}{\|z\|^a}d\nu<\infty\ \text{for some}\ a>0.
\end{equation}
The conditions (\ref{gr}) and (\ref{gr2}) ensure that the weighted measure $e^{-2n\varphi}d\nu$ has finite moments up to order $n$. Whereas condition (\ref{bm}) implies that $L^2$ and $sup$ norms of weighted polynomials  are  asymptotically equivalent. We also remark that BM-measures always exist (see eg. \cite{BMsurvey}). 

\subsection{Domination Principle}
Let $X=\Bbb{P}^m$ be the complex projective space and $\omega$ denotes the Fubini-Study K\"ahler form normalized by $\int_X\omega^m=1$. We also denote the set of all $\omega$-psh functions by
$$PSH(X,\omega)=\{\phi\in L^1(X)|\ \phi \ \text{usc}\ \text{and}\ \omega+dd^c\phi\geq 0\}.$$

Following \cite{GZ1}, we define \textit{non-pluripolar Monge-Amp\'ere} of $\phi\in PSH(X,\omega)$ by
\begin{equation} 
MA(\phi):=\lim_{j\to \infty}\{\mathbbm{1}_{\{\phi>-j\}}(\omega+dd^c\max[\phi,-j])^m\}.
\end{equation}
It follows from  \cite{GZ1} that the $MA(\phi)$ is a (positive) Borel measure satisfying 
\begin{equation}\label{min}
\int_X MA(\phi)\leq \int_X\omega^m=1.
\end{equation}

\begin{defn}
We denote
$$\mathcal{E}(X,\omega):=\{\phi\in PSH(X,\omega)\ | \int_XMA(\phi)=1\}.$$
\end{defn}
Then we have the following generalized domination principle due to Dinew \cite{Dinew}:

\begin{thm}\label{domp}
Let $\psi\in PSH(X,\omega)$ and $\phi\in \mathcal{E}(X,\omega)$ that satisfy $\psi\leq \phi$ a.e. with respect to $MA(\phi).$ Then $\psi\leq \phi$ on  $X$.
\end{thm}
It is well know that (see eg.\ \cite{DemBook}) there is a 1-1 correspondence between Lelong class psh function $\mathcal{L}(\Bbb{C}^m)$ and the set of $\omega$-psh functions which is given by the natural identification
\begin{equation}\label{LL}u\in \mathcal{L}(\C) \to \varphi(z):=
\begin{cases}
u(z)-\frac{1}{2}\log(1+\|z\|^2) & \text{for}\ z\in\C\\
\limsup_{w\in\C\to z}  u(w)-\frac{1}{2}\log(1+\|w\|^2) & \text{for}\ z\in H_{\infty}
\end{cases}  
 \end{equation}  where $\Bbb{P}^m=\Bbb{C}^m\cup H_{\infty}$ and $H_{\infty}$ denotes the hyperplane at infinity.

Now, writing $u\in \mathcal{L}(\Bbb{C}^m)$ as $u=\phi+u_0$ where $\phi\in PSH(X,\omega)$ and $u_0(z)=\frac12\log(1+\|z\|^2)$ on $\Bbb{C}^m$ we see that
$$NP(dd^cu)^m=MA(\phi)$$ on $\Bbb{C}^m$ where
\begin{equation}\label{npma}
NP(dd^cu)^m=\lim_{j\to \infty}\{\mathbbm{1}_{\{u>-j\}}(dd^c\max[u,-j])^m\}
\end{equation} denotes the non-pluripolar Monge-Amp\`ere (cf. \cite[\S 4]{BT87}). Hence, we obtain the following $\Bbb{C}^m$ version of Dinew's domination principle: 

\begin{cor}\label{dom}
Let $u,v\in \mathcal{L}(\Bbb{C}^m)$ and assume that 
$$\int_{\C}NP(dd^cu)^m=1.$$ If $v\leq u$ a.e with respect to $NP(dd^cu)^m$ then $v\leq u$ on $\Bbb{C}^m$.
\end{cor}

\subsection{Mass Asymptotics}\label{secMass} We continue with a basic result which asserts that mass equidistribution for sequences of weighted polynomials implies $L^1_{loc}$-convergence of potentials to the weighted global extremal function. 

\begin{thm}\label{basic} Let $Y\subset \C$ be a locally regular closed set, $\varphi:Y\to \Bbb{R}$ be a continuous weight function and $\nu$ be a BM-measure for the weighted set $(Y,\varphi)$. If $Y$ is unbounded, we also require $\varphi(z)$ to verify (\ref{gr}). We assume that 
\begin{equation}\label{npmass}
\int_{Int(S_{Y,\varphi})}d\mu_{Y,\varphi}=1.
\end{equation} 
Furthermore, let $p_n\in \pn$ be a sequence of polynomials verifying
\begin{equation}\label{an}
\limsup_{n\to \infty}\frac1n\log\|p_ne^{-n\varphi}\|_{L^2(\nu)}\leq 0
\end{equation}
and  assume that
\begin{equation}\label{a1}
\frac{1}{d_n}|p_n(z)|^2e^{-2n\varphi}d\nu\to \mu_{Y,\varphi}
\end{equation}
in the weak-star topology of measures on $S_{Y,\varphi}$. Then 
\begin{equation}
\frac1n\log|p_n|\to V_{Y,\varphi}\ \text{in}\ L^1_{loc}(\C).
\end{equation}
In particular,
$$\frac1n[Z_{p_n}]:=\frac1ndd^c\log|p_n|\to dd^c(V_{Y,\varphi})$$ in the sense of currents.
\end{thm}
 The hypothesis (\ref{a1}) means that for each continuous function $u\in C(S_{Y,\varphi})$ we have
 $$\frac{1}{d_n}\int_{S_{Y,\varphi}} u(z) |p_n(z)|^2e^{-2n\varphi(z)}d\nu\to \int_{S_{Y,\varphi}}  u(z) d\mu_{Y,\varphi}\ \text{as}\ n\to\infty.$$ We remark that the normalization factor $\frac{1}{d_n}$ is non-standard (cf. \cite{SZ,Zel}). However, the current version is more suitable for our purposes (cf. Theorem \ref{main}). In complex dimension one, $[Z_{p_n}]=\sum_{p_n(z)=0}\delta_z$ becomes counting measure on zeros of $p_n$. Hence, Theorem \ref{basic} gives a sufficient condition for zeros of weighted polynomials to be equidistributed with respect to the associated equilibrium measure. We also remark that assumption (\ref{npmass}) requires, in particular, that $Int(S_{Y,\varphi})$ is a non-empty open subset of $\C$. This is necessary as the following example shows: 
 
 \begin{example}In the spacial case $Y=S^1$ unit circle and $\varphi\equiv 0$ we have $V_{Y}(z)=\log^+|z|$ and $\mu_{Y,\varphi}=\frac{1}{2\pi}d\theta$ is the normalized arc-length measure. In this case, the monomials $p_n(z)=z^n$ and $\nu=\frac{1}{2\pi}d\theta$ satisfy the hypotheses of the Theorem \ref{basic} but $\frac1n\log|p_n(z)|=\log|z|\not=\log^+|z|$ in $L_{loc}^1(\Bbb{C})$.  
 \end{example}
 We thank Tom Bloom for pointing this example out. We are also grateful to N. Levenberg for his comments on an earlier draft. 
\begin{proof}[Proof of Theorem \ref{basic}]
We fix $r\gg1$ such that $V_{Y,\varphi}=V_{Y_r,\varphi}$ which implies that $S_{Y,\varphi}\subset Y_r$. Then by (\ref{an}) and BM inequality (\ref{bm}) for each $\epsilon>0$
$$\|p_ne^{-n\varphi}\|_{S_{Y,\varphi}}\leq \|p_ne^{-n\varphi}\|_{Y_r}\leq e^{\epsilon n}M_n$$ for sufficiently large $n$. It follows from Theorem 2.5 of \cite[Appendix B]{SaffTotik}, continuity of $\varphi$ and $V_{Y,\varphi}$ that $V_{Y,\varphi}=\varphi$ on $S_{Y,\varphi}$. This implies that
$$|p_n(z)|\leq  M_n e^{n(V_{Y,\varphi}(z)+\epsilon)}\ \forall z\in S_{Y,\varphi}.$$ 
Applying \cite[Theorem 2.6 in Appendix B]{SaffTotik} we deduce that 
\begin{equation}\label{beq1}
|p_n(z)|\leq  M_n e^{n(V_{Y,\varphi}(z)+\epsilon)}\ \forall z\in\C
\end{equation} for sufficiently large $n$.
Since $\epsilon>0$ arbitrary, by \cite[Theorem 5.2.1]{Klimek} we conclude that for every sequence of positive integers $J$ the function
\begin{equation}\label{c1}
G(z):=(\limsup_{n\in J}\frac1n\log|p_n(z)|)^*\in \mathcal{L}(\C)
\end{equation} and satisfies
\begin{equation}\label{c2}
G\leq V_{Y,\varphi}\ \text{on}\ \C.
\end{equation} Next, we claim that  
\begin{equation}\label{c3}
G(z)=V_{Y,\varphi}(z)\ \text{on}\ Int(S_{Y,\varphi}).
\end{equation}
Postponing the proof of the claim for the moment and assuming (\ref{c3}), since $V_{Y,\varphi}$ is locally bounded on $\C$ and $Int(S_{Y,\varphi})$ is an open subset of $\C$, by (\ref{npmass}) and (\ref{min}) we deduce that
\begin{equation}
\int_{\C}NP(dd^cG)^m=\int_{Int(S_{Y,\varphi})}NP(dd^cG)^m=\int_{Int(S_{Y,\varphi})}d\mu_{Y,\varphi}=1.
\end{equation} Here, $NP(dd^cG)^m$ denotes the non-pluripolar Monge-Amp\`ere of $G$. This implies that
\begin{equation}
G(z)=V_{Y,\varphi}(z)\ a.e.\ \text{with respect to}\ NP(dd^cG)^m.
\end{equation}  Thus, we can apply domination principle Corollary \ref{dom} with $u=G$ and $v=V_{Y,\varphi}$ to conclude that $$G=V_{Y,\varphi}\ \text{on}\ \C.$$ Hence, the theorem follows from \cite[Proposition 4.4]{BloomL}.

Now, we return the proof of the claim (\ref{c3}). To this end, assume that $G(w)<V_{Y,\varphi}(w)$ for some $w\in Int(S_{Y,\varphi}).$ We fix an open ball $w\in B\subset Int(S_{Y,\varphi})$. Note that by (\ref{beq1}) and \cite[Theorem 3.2.12]{Hormander} there are two options:
\begin{itemize}
\item[(i)] $\frac1n\log|p_n|\to -\infty$ locally uniformly on $B$ 
\item[(ii)] there exists a further subsequence $J_1$ such that for $n\in J_1$ $$\frac1n\log|p_n|\to g\ \text{in}\ L^1(B).$$ 
\end{itemize}
First, we rule out the option (i). Indeed, otherwise 
$$|p_n|e^{-n\varphi}\ll 1\ \text{locally uniformly on}\ B\ \text{for}\ n\gg1$$ which contradicts (\ref{a1}). Thus, (ii) occurs. Then passing to a further subsequence $J_2\subset J_1$ we conclude that 
$$\frac1n\log|p_n|\to g\ a.e.\ \text{on}\ B.$$ Note that $g^*$ is psh on $B$ and $g^*=g$ a.e on $B$ hence $g^*\leq G$ on $B$ which in turn implies that $g^*(w)<V_{Y,\varphi}(w)$. Then by Hartogs' lemma and continuity of $V_{Y,\varphi}$ there exists $\delta,\rho>0$ such that $B(w,\rho)\subset B$ and 
$$\frac1n\log|p_n(z)|<V_{Y,\varphi}(z)-\delta,\ \forall z\in B(w,\rho)$$ for large $n\in J_2.$ Since $V_{Y,\varphi}\leq \varphi$ on $Y$ we infer that
$$|p_n(z)|e^{-n\varphi(z)}\leq e^{-n\delta}\ \ \forall z\in B(w,\rho)$$ for large $n\in J_2.$ This contradicts (\ref{a1}). Hence, we conclude that
$$g^*(w)=V_{Y,\varphi}(w)\leq G(w)\ \text{for}\ w\in Int(S_{Y,\varphi}).$$
This finishes the proof.

\end{proof}

\section{Mass Asymptotics of Random Polynomials}\label{S3}
In the rest of this paper we consider the special case where $Y=\C$ and $\varphi:\C\to \Bbb{R}$ is a $\mathscr{C}^{1,1}$ function. We also assume that $\varphi$ verifies (\ref{growth}). We denote the corresponding global extremal function
\begin{equation}\label{ex}
\varphi_e(z):=V_{\C,\varphi}(z)=\sup\{\psi(z):\psi\in \mathcal{L}(\C), \psi\leq \varphi\ \text{on}\  \C\}
\end{equation}
and the support $S_{\varphi}:=S_{\C,\varphi}$ of the Monge-Amp\`ere $\mu_{\varphi_e}:=\frac{1}{m!}(dd^cV_{\C,\varphi})^m$. 

 In \cite[Corollary 3.6]{Berman} Berman proved that  
  
\begin{equation}
S_{\varphi}:=\{z\in\C: \varphi(z)=\varphi_e(z)\ \text{and}\ dd^c\varphi(z)>0 \}.
\end{equation}
and 
\begin{equation}\label{eqm}
\mu_{\varphi_e}=\mathbbm{1}_{S_{\varphi}}\det(dd^c\varphi)dV_m.
\end{equation} 
We remark that by $\mathscr{C}^{1,1}$ regularity $dd^c\varphi(z)=\frac{i}{\pi}\sum_{j,k} \frac{\partial^2\varphi}{\partial z_j\partial\overline{z}_k}dz_j\wedge d\overline{z}_k$  is well-defined at Lebesgue almost every $z\in\C$ and the condition $dd^c\varphi(z)>0$ implies all eigenvalues of the Hessian $\begin{bmatrix} \frac{\partial^2\varphi}{\partial z_j\partial\overline{z}_k}\end{bmatrix}_{j,k}$ are positive. Moreover, $\det(dd^c\varphi):=(\frac{2}{\pi})^m\det\begin{bmatrix} \frac{\partial^2\varphi}{\partial z_j\partial\overline{z}_k}\end{bmatrix}$.

\subsection{Bergman Kernel Asymptotics}
 For a fixed orthonormal basis (ONB) $\{P^n_j\}_{j=1}^{d_n}$ for $\pn$ with respect to the norm (\ref{in}) the Bergman kernel is given by
$$K_n(z,w):=\sum_{j=1}^{d_n}P_j^n(z)\overline{P_j^n(w)}.$$ 
We also denote the \textit{Bergman function} by
$$B_n(z):=K_n(z,z)e^{-2n\varphi(z)}=\sum_{j=1}^{d_n}|P^n_j(z)|^2e^{-2n\varphi(z)}.$$
Bergman function $B_n$ has the extremal property
\begin{equation}\label{pick}
B_n(z)=\sup_{f_n\in \pn\setminus \{0\}}\frac{|f_n(z)|^2e^{-2n\varphi(z)}}{\|f_n\|_{n}^2}.
\end{equation} where $\|f_n\|_{n}$ denotes the norm induced by (\ref{in}). 
Moreover, we have the following dimensional density property
$$\int_{\C}B_n(z)dV_m(z)=\dim(\pn)=O(n^m).$$

The following result will be useful in order to obtain expected mass distribution of random polynomials (see Proposition \ref{exp}). 

\begin{thm}\cite{Berman1}\label{Ber2}
Let $\varphi:\C\to \Bbb{R}$ be a $\mathscr{C}^{1,1}$-weight function satisfying (\ref{growth}). Then
$$n^{-m}K_n(z,z)e^{-2n\varphi(z)} \to \mathbbm{1}_{S_{\varphi}}\det(dd^c\varphi)$$ in $L^1(\C)$. In particular, $n^{-m}K_n(z,z)e^{-2n\varphi(z)}dV_m(z)$ converges to the weighted equilibrium measure $\mu_{\varphi_e}$ in the weak-star topology on $\C$.
\end{thm}

The next result is also due to Berman \cite[Theorem 3.8]{Berman1} which allows us to get asymptotic Hilbert-Schmidt norms of the Toeplitz operators (see Proposition \ref{prop2}):

 \begin{thm}\label{Ber4}
 Let $\varphi:\C\to \Bbb{R}$ be a $\mathscr{C}^{1,1}$-weight function satisfying (\ref{growth}). Then
 $$n^{-m}|K_n(z,w)|^2e^{-2n\varphi(z)-2n\varphi(w)}dV_m(z)dV_m(w)\to \Delta\wedge \mathbbm{1}_{S_{\varphi}}\mu_{\varphi_e} $$ as measures on $\C\times\C$ in weak-star topology. 
  \end{thm}
Here; $\Delta:=[\{z=w\}]$ denotes the current of integration along the diagonal in $\C\times\C$ and for any bounded continuous function $\Psi$ we have 
 $$\int_{\C\times \C} \Psi(x,y) \Delta\wedge \mathbbm{1}_{S_{\varphi}}\mu_{\varphi_e}:= \int_{S_{\varphi}} \Psi(x,x) d\mu_{\varphi_e}.$$ 

\subsection{Toeplitz operators and limiting distribution of eigenvalues}
We denote the orthogonal projection $$\Pi_n:L^2(\Bbb{C}^m, e^{-2n\varphi(z)}dV_m)\to \pn$$ onto the finite dimensional subspace $\pn$. For a bounded function $g:\Bbb{C}^m\to\Bbb{R}$ we also let $$M_g: L^2(\Bbb{C}^m, e^{-2n\varphi(z)}dV_m)\to L^2(\Bbb{C}^m, e^{-2n\varphi(z)}dV_m)$$ denote \textit{multiplication operator} defined by $$M_g(h)(z)=g(z)h(z).$$ We consider the sesquilinear form on $\pn$ defined by
$$\langle p,q\rangle_g:=\int_{\Bbb{C}^m}g(z)p(z)\overline{q(z)}e^{-2n\varphi(z)}dV_m.$$ Then by linear algebra there is a self-adjoint operator $T_n^g:\pn\to \pn$ such that
$$\langle p,q\rangle_g=\langle T_n^gp,q\rangle_n.$$
Note that $T_n^gp$ is nothing but the composition of orthogonal projection with the multiplication operator on $\pn$ i.e.\
$$T_n^g=\Pi_n\circ M_g$$ which is called $n^{th}$ \textit{Toeplitz operator} with multiplier $g$. The latter property implies that
$$T_n^gp(z)=\int_{\Bbb{C}^m}g(w)p(w)K_n(z,w)e^{-2n\varphi(w)}dV_m(w).$$ 
The following is a standard result in this setting and it indicates a connection between the Toeplitz operators and mass equidistribution:
\begin{prop}\label{prop2} Let $\varphi:\C\to \Bbb{R}$ be a $\mathscr{C}^{1,1}$-weight function satisfying (\ref{growth}) and $g:\C\to\Bbb{R}$ be a bounded function. Then
\begin{enumerate}

\item $Tr(T_n^g)=\int_{\Bbb{C}^m}g(z)K_n(z,z)e^{-2n\varphi(z)}dV_m.$

\item For each $k\in \Bbb{N}$ we have $$\frac{1}{d_n}Tr((T_n^g)^k)\to \int_{\Bbb{C}^m}g^k(z)d\mu_{\varphi_e}$$ as $n\to\infty.$
\end{enumerate} 
\end{prop}

\begin{proof}
(1) Note that $T_n^g$ admits an ONB of eigenvectors $\{p_j^n\}_{j=1}^{d_n}.$ Letting 
\begin{equation}\label{ev}
\mu_j:=\langle T_n^gp^n_j,p^n_j\rangle_n=\langle p_j^n,p_j^n\rangle_g
\end{equation}
we obtain
$$Tr(T_n^g)=\sum_{j=1}^{d_n}\int_{\Bbb{C}^m}g(z)|p_j^n(z)|^2e^{-2n\varphi(z)}dV_m=\int_{\Bbb{C}^m}g(z)K_n(z,z)e^{-2n\varphi(z)}dV_m.$$

(2) It follows from Theorem \ref{Ber2} that
$$\frac{1}{d_n}Tr(T_n^g)\to \int_{\Bbb{C}^m}g(z)d\mu_{\varphi_e}.$$
Note that $(T_n^g)^2=\Pi_nM_g\Pi_nM_g$ and
$$Tr((T_n^g)^2)=\int_{\Bbb{C}^m}\int_{\Bbb{C}^m}g(z)g(w)|K_n(z,w)|^2e^{-2n(\varphi(z)+\varphi(w))}dV_m(z)dV_m(w).$$
Hence, by Theorem \ref{Ber4} we have
$$\frac{1}{d_n}Tr((T_n^g)^2)\to \int_{\Bbb{C}^m}\int_{\Bbb{C}^m}g(z)g(w) \Delta\wedge \mathbbm{1}_{S_{\varphi}}\mu_{\varphi_e}=\int_{\Bbb{C}^m} g^2(z) d\mu_{\varphi_e}.$$
Now, for $k\geq 3$ we have
$$\mu_j^k=\langle (\Pi_nM_g)^kp_j^n,p_j^n\rangle_n$$
and hence,
$$\sum_{j=1}^{d_n}\mu_j^k=\int_{\Bbb{C}^m}\int_{\Bbb{C}^m}g(z)g(w)^{k-1}|K_n(z,w)|^2e^{-2n\varphi(z)-2n\varphi(w)}dV_m(z)dV_m(w).$$
Thus, it follows from Theorem \ref{Ber4} that
$$\frac{1}{d_n}Tr((T_n^g)^k)\to \int_{\Bbb{C}^m}g^k(z)d\mu_{\varphi_e}.$$
\end{proof}

\subsection{Subgaussian Random Variables}\label{gauss1}

In this section we recall basic properties of subgaussian random variables. Let $(\Omega,\mathcal{F},\tau)$ be a probability space. A real valued random variable $X:\Omega\to\R$ is called \textit{subgaussian} with parameter $b>0$ (or $b$-subgaussian) if the moment generating function (MGF) of $X$ is  dominated by MGF of normalized Gaussian $N(0,b)$ that is
\begin{equation}\label{sgd}
\E[e^{tX}]\leq e^{\frac{b^2t^2}{2}}\ \text{for all}\ t\in\R.
\end{equation}
We remark that the above definition is non-standard (cf. \cite[\S 5.2.3]{Vershynin}); in particular (\ref{sgd}) forces that $\E[X]=0$ which is a convenient assumption for our setting. The classical examples of 1-subgaussian random variables are Standard Gaussian $N(0,1)$, Bernoulli random variables $\Bbb{P}[X=\pm 1]=\frac12$, and uniform distribution on $[-1,1]$. Moreover, all bounded random variables of mean zero are subgaussian. More precisely, if $\E[X]=0$ and $X\leq b$ almost surely then $X$ is $b$-subgaussian. We have the following characterization of subgaussian random variables.

\begin{prop}\cite[Lemma 5.5]{Vershynin}\label{sgb}
Let $X$ be a centered real random variable (i.e.\ $\E[X]=0$). Then the following are equivalent:
\begin{enumerate}
\item $\exists b>0$ such that $\E[e^{tX}]\leq e^{\frac{b^2t^2}{2}}\ \text{for all}\ t\in\R$.
\item $\exists c>0$ such that $\Bbb{P}[|X|>\alpha]\leq 2e^{-c\alpha^2}$ for every $\alpha>0$.
\item $\exists K>0$ such that $(\E[|X|^p)^{\frac1p}\leq K\sqrt{p}$ for all $p\geq 1$.
\item $\exists \kappa>0$ such that $\E[e^{X^2/\kappa^2}]\leq 2$.
\end{enumerate}
\end{prop}
The last property is known as $\psi_2$ condition. More precisely, a centered random variable $X$ is subgaussian if and only if its \textit{Orlicz norm}
\begin{eqnarray}
\|X\|_{\psi_2}:&=&\inf_{\kappa>0}\{\E[e^{X^2/\kappa^2}]\leq 2\}
\end{eqnarray}
is finite.

\subsubsection{Hanson-Wright Inequality}
Let $X_j$ be independent subgaussian random variables and $\kappa_j:=\|X_j\|_{\psi_2}$. We denote the joint probability distribution of $X:=(X_1,\dots, X_N)$ by $\Bbb{P}$. We also let $A=[A_{ij}]$ be a square matrix with real entries. We denote its operator norm 
$$\|A\|:=\max_{\|v\|_2\leq 1}\|Av\|$$ where $\|\cdot\|_2$ denotes Euclidean norm and the Hilbert-Schmidt norm by
$$\|A\|_{HS}:=(\sum_{i,j}|a_{ij}|^2)^{1/2}=[Tr(AA^T)]^{1/2}.$$
We consider the random quadratic form
$$X\to X^TAX.$$
The following concentration inequality goes back to Hanson-Wright \cite{HW}. The version we use here is due to Rudelson-Vershynin \cite{RuVe}:
\begin{thm}[Hanson-Wright Inequality]\label{HWI}
Let $A$ be a $N\times N$ square matrix and $X=(X_1,\dots,X_N)\in\R^N$ be a random vector whose components $X_j$ are independent subgaussian variables such that $$\|X_j\|_{\psi_2}\leq K$$ for $j=1,\dots,N$. Then for each $t\geq 0$ 
$$\Bbb{P}[|X^TAX-\E[X^TAX]|>t]\leq 2\exp\big(-c\min\{\frac{t^2}{K^4\|A\|_{HS}^2},\frac{t}{K^2\|A\|}\}\big)$$ 
where $c>0$ is an absolute constant which does not depend on $t$.
\end{thm}

\subsubsection{Complex Case}
Next, we formulate Hanson-Wright inequality for complex random variables and Hermitian matrices with complex entries. Let $X:\Omega\to\Bbb{C}$ be a complex valued random variable. We denote the real and imaginary parts of $X$ by $Re(X)$ and $Im(X)$ respectively. We say that $X$ is subgaussian if $Re(X)$ and $Im(X)$ are independent subgaussian random variables. For a Hermitian square matrix $A$ we let 
$$\tilde{A}=\begin{bmatrix} Re(A) & -Im(A) \\ Im(A) & Re(A)\end{bmatrix}$$ where $Re(A):=[Re(a_{ij})]$ and $Im(A)=[Im(a_{ij})]$. Under these definitions we have $\|\tilde{A}\|=\|A\|$ and $\|\tilde{A}\|_{HS}=\sqrt{2}\|A\|_{HS}$. We also let $X^*:=\overline{X}^T$. Then applying Theorem \ref{HWI} we obtain:

\begin{thm}[Complex Hanson Wright Inequality] \label{HWII}Let $A$ be a $N\times N$ Hermitian matrix and $X=(X_1,\dots,X_N)\in\Bbb{C}^N$ be a random vector whose components $X_j$ are independent complex subgaussian random variables such that $\max_j(\|Re(X_j)\|_{\psi_2},\|Im(X_j)\|_{\psi_2})\leq K$. Then there exists an absolute constant $c>0$ such that
$$\Bbb{P}[|X^*AX-\E[X^*AX]|>t]\leq 2\exp\big(-c\min\{\frac{t^2}{K^4\|A\|_{HS}^2},\frac{t}{K^2\|A\|}\}\big)$$ for each $t\geq0$.
\end{thm}

\begin{rem}\label{rem}
Finally, we remark that in case $A$ is Hermitian matrix and $X_j$ are real valued subgaussian by setting $
\tilde{A}:=\begin{bmatrix} Re(A)\\ Im(A)\end{bmatrix}$ the corresponding concentration inequality follows from \cite[Theorem 2.1]{RuVe}. 
\end{rem}

\subsection{Mass Equidistribution for Random Polynomials}\label{S4}
 For a fixed ONB $\{P_j^n\}_{j=1}^{d_n}$ of $\mathcal{P}_n$ with respect to the norm induced by (\ref{in}) we consider \textit{subgaussian random polynomials}
$$f_n(z)=\sum_{j=1}^{d_n}c^n_jP_j^n(z)$$ where $c^n_j$ are independent identically distributed (iid) real or complex subgaussian random variables of mean zero and unit variance i.e.\ $\E[|c^n_j|^2]=1$. We endow the vector space $\mathcal{P}_n$ with the $d_n$-fold product probability measure $Prob_n$ induced by the law of $c_j^n$. We also consider the product probability space $\prod_{n=1}^{\infty}(\pn,Prob_n)$ whose elements are sequences $(f_1,f_2,\dots)$ of random polynomials.

Let $g:\C\to\Bbb{R}$ be a bounded function, in what follows we consider the random variables
$$X_n^g:\pn\to \Bbb{R}$$
\begin{eqnarray*}
X_n^g(f_n)&=&\int_{\Bbb{C}^m}g(z)|f_n(z)|^2e^{-2n\varphi(z)}dV_m\\
&=& \langle T_n^g(f_n),f_n\rangle_n.
\end{eqnarray*}
Next, we obtain asymptotic expected value of $X_n^g$:

\begin{prop}\label{exp}
Assume that random coefficients $c_j^n$ are iid real or complex subgaussian random variables of mean zero and unit variance. Then
$$\Bbb{E}[X_n^g]=Tr(T_n^g).$$
In particular, $$\frac{1}{d_n}\Bbb{E}[X_n^g] \to \int_{\Bbb{C}^m} g(z)d\mu_{\varphi_e}\ \text{as}\ n\to \infty.$$
\end{prop}

\begin{proof}
Note that $$|f_n(z)|^2=\sum_{1\leq j,k\leq d_n}c_j\overline{c_k}P_j(z)\overline{P_k(z)}.$$ Since $c_j^n$ are iid of mean zero and unit variance, we have
$\E[|f_n(z)|^2]=K_n(z,z)$ for every $z\in\C$. Thus, by Fubini's Theorem
\begin{equation}
\Bbb{E}[X_n^g] =\int_{\C} g(z)K_n(z,z)e^{-2n\varphi(z)}dV_m= Tr(T_n^g).
\end{equation}
Hence, the second assertion follows from Proposition \ref{prop2}.
\end{proof}

\begin{proof}[Proof of Theorem \ref{main}]
We prove the case where $c_j^n$ are iid complex subgaussian random variables. The real case follows from the same argument and Remark \ref{rem}.
 
For the fixed ONB basis of eigenfunctions $\{P_j^n\}_{j=1}^{d_n}$ we may identify the random polynomials 
$$f_n=\sum_{j=1}^{d_n}c_j^nP_j^n$$ with the subgaussian random vector 
\begin{equation}
Z_n:=(c_1^n,\dots,c_{d_n}^n)\in \Bbb{C}^{d_n}
\end{equation}
and denote the probability law of $Z_n$ by $\Bbb{P}_n$. We also denote Euclidean norm of $Z_n$ by $\|Z_n\|$. First, we observe that
\begin{equation}\label{nc}
\Bbb{P}_n[\big\{Z_n\in\Bbb{C}^{d_n}: \|Z_n\|\leq d_n\ \text{for sufficiently large $n$}\big\}]=1.
\end{equation}
Indeed, by Proposition \ref{sgb} (2) there exists $b>0$ (independent of $n$) such that
\begin{eqnarray}
\Bbb{P}_n[\big\{Z_n\in\Bbb{C}^{d_n}: \|Z_n\|> d_n\big\}] & =& \Bbb{P}_n[\big\{c_j^n\in\Bbb{C}^{d_n}: \sum_{j=1}^{d_n}|c_j^n|^2> d_n^2\big\}] \nonumber\\
&\leq & \Bbb{P}_n[\big\{c_j^n\in\Bbb{C}^{d_n}: |c_j^n|^2>d_n\ \text{for some $j$}\big\}] \nonumber \\
&\leq& 2d_n exp(-bd_n). \label{bc1}
\end{eqnarray}
Since the right hand side of the last inequality (\ref{bc1}) is summable over $n$, the claim follows from Borel-Cantelli lemma.

Next, we identify the operator $T_n^g$ with a Hermitian $d_n\times d_n$ matrix $A_n^g.$ Note that with the new notation we have
$$X_n^g(f_n)=\langle A_n^gZ_n,Z_n\rangle$$ where $\langle,\rangle$ denotes the standard Hermitian inner product on $\Bbb{C}^{d_n}$. Then by Theorem \ref{HWII} there exists an absolute constant $c>0$ such that

\begin{eqnarray}
Prob_n[|X_n^g(f_n)-\E[X_n^g]|>t] &=& \Bbb{P}_n[|Z_n^*A_n^gZ_n-\E[Z_n^*A_n^gZ_n]|>t]\\
&\leq& 2\exp\big(-c\min\{\frac{t^2}{K^4\|A_n^g\|_{HS}^2},\frac{t}{K^2\|A_n^g\|}\}\big)
\end{eqnarray}
where $K:=\|c_j^n\|_{\psi_2}\geq1$. Note that $\displaystyle\|A_n^g\|\leq \sup_{z\in\C}|g(z)|$ and by Proposition \ref{prop2} we have $$\|A_n^g\|_{HS}^2=Tr((T_n^g)^2)=O(d_n).$$ Letting $t=\epsilon d_n$, by Theorem \ref{Ber2} we deduce that for sufficiently large $n$
$$Prob_n[|\frac{1}{d_n}X_n^g(f_n)-\int_{\C}g(z)d\mu_{\varphi_e}|>\epsilon]\leq 2\exp(-C_gd_n)$$ where $C_g>0$ is an absolute constant that deos not depend on $n$.
Hence, it follows from Borel-Cantelli lemma that there exists a set $\mathcal{A}_g\subset\prod_{n=1}^{\infty}\pn$ of probability one such that
$$\frac{1}{d_n}X_n^g(f_n) \to \int gd\mu_{\varphi_e}$$ for every $\{f_n\}\in \mathcal{A}_g.$
Next, we let $\{g_j\}_{j=1}^{\infty}$ be a countable dense subset of $C(S_{\varphi})$ and define 
\begin{equation}
\mathcal{A}:=\cap_{j=1}^{\infty}\mathcal{A}_{g_j}\cap\big\{(f_n)\in \prod_{n=1}^{\infty}\pn: \|p_n\|_n^2\leq d_n\ \text{for sufficiently large}\ n\big\}.
\end{equation}
By (\ref{nc}) and being countable intersection of sets with probability one, $Prob(\mathcal{A})=1$. Moreover, for each $\{f_n\}\in \mathcal{A}$ we have
\begin{equation}
\frac{1}{d_n}|f_n(z)|^2e^{-2n\varphi(z)}dV_m \to d\mu_{\varphi_e}
\end{equation}
as measures on $S_{\varphi}$. Indeed, for each $u\in C(S_{\varphi})$ and $\epsilon>0$ choose $g_j$ such that
$\|u-g_j\| _{S_\varphi}\leq \epsilon$. Then for sufficiently large $n$ we have
$$|\frac{1}{d_n}X_n^{g_j}(f_n)-\int g_jd\mu_{\varphi_e}|\leq \epsilon $$
hence,
\begin{eqnarray*}
|\frac{1}{d_n}\int_{S_{\varphi}}u(z)|f_n(z)|^2e^{-2n\varphi(z)}dV_m &-& \int_{S_{\varphi}}u(z)d\mu_{\varphi_e}|\\
 &\leq& \|u-g_j\| _{S_\varphi}( \frac{1}{d_n}\|f_n\|_n^2 +1)+ |\frac{1}{d_n}X_n^{g_j}(f_n)-\int g_jd\mu_{\varphi_e}| \\
&\leq& 3\epsilon.
\end{eqnarray*}
This proves the first assertion.

Since the hypotheses of Theorem \ref{basic} hold with probability one we obtain the second assertion.
\end{proof}

\section{Random Orthonormal Bases and Distribution of Zeros}\label{S5}
In this section, we consider \textit{random orthogonal polynomials}. More precisely, for a fixed ONB $\{P^n_j\}_{j=1}^{d_n}$ for $\pn$ with respect to the norm (\ref{in}) we may identify each ONB $\mathcal{B}=\{F_j^n\}_{j=1}^{d_n}$ for $\pn$ with a unitary matrix $U_{\mathcal{B}}\in \mathcal{U}(d_n).$ Thus, we consider the set of of all orthonormal bases for $\pn$ as a probability space by identifying it with the unitary ensemble $ \mathcal{U}(d_n)$ endowed with the Haar probability measure. Moreover, we let $\mathcal{ONB}:=\prod_{n\geq1} \mathcal{U}(d_n)$ be the product probability space. In this context, we have the following result (cf. \cite[Theorem 1.2]{SZ} see also \cite{Zel}):

\begin{thm}\label{onb}
For almost every sequence of ONB $\mathcal{B}=\{F_j^n\}$ in $\mathcal{ONB}$ there exists a subsequence $\Lambda_n\subset \{1,\dots,d_n\}$ of density one (i.e. $\frac{\#\Lambda_n}{d_n}\to 1$ as $n\to \infty$) such that
$$\lim_{\substack{n\to \infty\\ j\in \Lambda_n}}\int_{\Bbb{C}^m}g(z)|F_j^n(z)|^2e^{-2n\varphi(z)}dV_m=\int_{\Bbb{C}^m} g(z)d\mu_{\varphi_e}$$ for every bounded continuous function $g:\C\to \R$. If $m\geq2$ then the entire sequence has this property.
\end{thm}

\begin{proof}
We consider the random variables 
$$\mathcal{Y}_n^g:\mathcal{ONB}\to \Bbb{R}$$
$$\mathcal{Y}_n^g(\mathcal{B}):=\sum_{j=1}^{d_n}|\langle U_{\mathcal{B}}^*A_n^gU_{\mathcal{B}}e_j,e_j\rangle|^2$$ where $A_n^g$ is the matrix representing the Toeplitz operator $T_n^g$ with symbol $g$ and $e_j$ is the standard basis element whose $j^{th}$ coordinate is 1. By (\ref{ev}) and invariance of Haar measure under left-multiplication with a unitary matrix
\begin{eqnarray}
\Bbb{E}[\langle U^*_{\mathcal{B}}A_n^gU_{\mathcal{B}}e_j,e_j\rangle] &=&  \int_{\mathcal{U}(d_n)}(\langle U^*A_n^gUe_j,e_j\rangle)dU \nonumber\\
&=& \sum_{j=1}^{d_j}\mu_j\E[|U_{1j}|^2] \nonumber\\
&=&\frac{1}{d_n}Tr(A_n^g).
\end{eqnarray} 
Next, we consider the standardized random variables 
\begin{eqnarray*}
\overline{\mathcal{Y}}_n^g(\mathcal{B}):&=&\sum_{j=1}^{d_n}|\langle U_{\mathcal{B}}^*A_n^gU_{\mathcal{B}}e_j,e_j\rangle-\frac{1}{d_n}Tr(T_n^g)|^2\\
&=& \sum_{j=1}^{d_n}|\langle U_{\mathcal{B}}^*\tilde{A}_n^gU_{\mathcal{B}}e_j,e_j\rangle|^2\\
\end{eqnarray*}
where $\tilde{A}_n^g=A_n^g-\frac{1}{d_n}Tr(T_n^g)I_n$ is of trace zero. 

Then by \cite[Lemma 5.1]{Zel} we obtain
\begin{eqnarray*}
\Bbb{E}[\overline{\mathcal{Y}}_n^g] 
&=&\sum_{j=1}^{d_n}\Bbb{E}[|\langle U^*_{\mathcal{B}}\tilde{A}_n^gU_{\mathcal{B}}e_j,e_j\rangle|^2]\\
&=& \int_{\Bbb{C}^m}g^2d\mu_{\varphi_e}-(\int_{\Bbb{C}^m}gd\mu_{\varphi_e})^2+o(1) \ \text{as}\ n\to\infty.
\end{eqnarray*}
This implies that
\begin{equation}\label{aex}
\lim_{N\to\infty}\frac1N\sum_{n=1}^N\Bbb{E}[\frac{1}{d_n}\overline{\mathcal{Y}}_n^g]\to 0
\end{equation}
since $\frac1N\sum_{n=1}^N\frac{1}{d_n}\to 0$ as $N\to\infty$.

On the other hand, since $g$ is bounded continuous function we have
 $$|\langle U_{\mathcal{B}}^*A_n^gU_{\mathcal{B}}e_j,e_j\rangle|=|\int_{\C}g|P_j^n|^2e^{-2n\varphi}dV_m|\leq \sup_{\C}|g|$$ which implies that
$$Var[\frac{1}{d_n}\overline{\mathcal{Y}}_n^g]\leq \sup_{\mathcal{B}}(\frac{1}{d_n}\overline{\mathcal{Y}}_n^g(\mathcal{B})\big)^2=O(1)$$ where the implied constant depends on $g$ but independent of $n$.  Since $\frac{1}{d_n}\overline{\mathcal{Y}}_n^g$ are independent random variables whose variances are bounded it follows from (\ref{aex}) and Kolmogorov's law of large numbers that as $N\to \infty$
$$\frac1N\sum_{n=1}^N\frac{1}{d_n}\overline{\mathcal{Y}}_n^g\to 0$$ almost surely. Thus, the first assertion follows from \cite[Theorem 1.20]{Walters}.

For the second assertion, note that for $m\geq2$ we have $\Bbb{E}[\frac{1}{d_n}\overline{\mathcal{Y}}_n^g]=O(\frac{1}{n^m})$ which in turn implies that $\Bbb{E}[\sum_{n=1}^{\infty}\frac{1}{d_n}\overline{\mathcal{Y}}_n^g]<\infty$ and hence $\frac{1}{d_n}\overline{\mathcal{Y}}_n^g\to 0$ almost surely.

\end{proof}

In the unweighted case \cite{Bloom4} T. Bloom proved that for every regular compact set $K\subset \Bbb{C}^m$ and Bernstein-Markov measure $\nu$, every ONB $\mathcal{B}=\{F_j^n\}\in\mathcal{ONB}$ has the property that
$$V_K(z)=\big(\limsup_{\substack{n\to \infty\\ j\in \{1,\dots,d_n\}}}\frac1n\log|F_j^n(z)|\big)^*\ \text{for all}\ z\in \Bbb{C}^m\setminus \hat{K}$$ where $\hat{K}$ denotes the polynomial convex hull of $K$. On the other hand, by the proof of  Theorem \ref{basic} an immediate consequence of Theorem \ref{onb} is that for almost every ONB $\mathcal{B}=\{F_j^n\}\in\mathcal{ONB}$
$$\varphi_e(z)=\big(\limsup_{\substack{n\to \infty\\ j\in \Lambda_n}}\frac1n\log|F_j^n(z)|\big)^*\ \text{for all}\ z\in \Bbb{C}^m.$$
However, we remark that Theorem \ref{onb} is a probabilistic result and the set of ONB which do not fall in its context is non-empty. For example in dimension one, for $\varphi(z)=\frac{|z|^2}{2}$ the $F_j(z)=\sqrt{\frac{n^{j+1}}{\pi j!}}z^j$ form an ONB for $\pn$ with respect to the norm $\|\cdot\|_n$. However, zeros of $F_j$ are not equidistributed with respect to the equilibrium measure. 

\section{Further Generalizations}\label{lineb}
 In the last part of this work we describe a generalization of Theorem \ref{main} to the line bundle setting. Let $M$ be a compact complex projective Hermitian manifold and $L \to M$ be an ample holomorphic line bundle endowed with a smooth (at least $\mathscr{C}^2)$ Hermitian metric $h=e^{-\varphi}$ where $\varphi=\{\varphi_{\alpha}\}$ is a local weight of the metric. The latter means that if $e_{\alpha}$ is a holomorphic frame for $L$ over an open set $U_{\alpha}$ then $|e_{\alpha}|_h=e^{-\varphi_\alpha}$ where $\varphi_{\alpha}\in \mathscr{C}^2(U_{\alpha})$ such that $\varphi_{\alpha}=\varphi_{\beta}+\log|g_{\alpha\beta}|$ and $g_{\alpha\beta}:=e_{\beta}/e_{\alpha}\in\mathcal{O}^*(U_{\alpha}\cap U_{\beta})$ are the transition functions for $L$. Then one can define \textit{global extremal weight} $\varphi_e$ to be 
 \begin{equation}\label{def}
 \varphi_{e}:=\sup\{\psi\ \text{is a psh weight}:  \psi\leq \varphi\ \text{on}\ M\}.
 \end{equation}
It follows that $\varphi_e$ defines a psh weight of the Hermitian metric $h_e:=e^{-\varphi_e}$ on $L$. We denote its curvature current by $dd^c\varphi_e:=dd^c(\varphi_{e,\alpha})$ on $U_{\alpha}$. Note that by the compatibility condition we have $\varphi_{e,\alpha}=\varphi_{e,\beta}+\log|g_{\alpha\beta}|$ and the current $dd^c\varphi_e$ is a globally well-defined positive closed $(1,1)$ current on $M$. Moreover, by \cite{Berman} the \textit{equilibrium measure} 
$$\mu_{\varphi_e}:=(dd^c\varphi_e)^m/m!$$ is supported on the compact set
$$S_{\varphi}:=M_{\varphi}(0)\cap D$$ where 
$M_{\varphi}(0):=\{x\in M: dd^c\varphi(x)>0\}$ and  
$D:=\{x\in M: \varphi(x)=\varphi_e(x)\}.$

The geometric data given above allow us to define a scalar inner product on the vector space of \textit{global holomorphic sections} $H^0(M,L^{\otimes n})$ via
\begin{equation}\label{inp}\langle s_1,s_2\rangle:=\int_X \langle s_1(x),s_2(x)\rangle_{h^{\otimes n}} dV
\end{equation} where $dV$ is a fixed volume form on $M$. We also denote the induced norm by $\|\cdot\|_n$. Next, we fix an ONB $\{S_j^n\}_{j=1}^{d_n}$ for $H^0(M,L^{\otimes n})$ with respect to the inner product (\ref{inp}). Then a \textit{subgaussian random holomorphic section} is of the form
$$S_n:=\sum_{j=1}^{d_n}c_j^nS_j^n$$ where $c_j^n$ are iid (real or complex) subgaussian random variables. This definition induces a $d_n$-fold product probability measure $Prob_n$ on the vector space $H^0(M,L^{\otimes n})$. We also consider the product probability space $\prod_{n=1}^{\infty}\big(H^0(M,L^{\otimes n}), Prob_n\big)$. The arguments in \S 3 carries over to the current geometric setting, in particular almost every sequence of subgaussian random holomorphic sections is quantum ergodic in the sense of \cite{Zel}:

\begin{thm}
Let $M$ be a projective complex manifold and $(L,h)$ be an ample Hermitian holomorphic line bundle endowed with a $\mathscr{C}^2$ metric $h$. Then for almost every sequence in $\prod_{n=1}^{\infty}\big(H^0(M,L^{\otimes n}), Prob_n\big)$  the masses 
\begin{equation}
\frac{1}{d_n}|s_n(z)|^2_{h^{\otimes n}}dV \to d\mu_{\varphi_e}
\end{equation} in the weak-star sense on $S_{\varphi}$.
 Moreover, almost surely in $\prod_{n=1}^{\infty}\big(H^0(M,L^{\otimes n}), Prob_n\big)$ the normalized currents of integration
 $$\frac1n[Z_{s_n}] \to dd^c\varphi_e$$ in the sense of currents.

\end{thm} 


\end{document}